\documentclass[12pt]{amsart}

\let\lable\label
\def\label#1{\marginpar{#1}\lable{#1}}

\usepackage{amsmath}
\usepackage{amssymb}

\theoremstyle{plain}
\newtheorem{theorem}{Theorem}[section]
\newtheorem{lemma}{Lemma}[section]
\newtheorem{proposition}{Proposition}[section]
\newtheorem{corollary}{Corollary}[section]

\theoremstyle{definition}
\newtheorem{definition}{Definition}[section]

\newtheorem{question}{Question}[section]

\theoremstyle{remark}
\newtheorem{remark}{Remark}

\numberwithin{equation}{section}

\DeclareMathOperator{\fix}{fix}

\DeclareMathOperator{\dom}{dom}

\DeclareMathOperator{\ran}{range}

\newcommand{\tie}[1]
{\,\raise-5pt\hbox{
${\buildrel{\displaystyle{\rhd}\!\!{\lhd}}\over{
\scriptstyle
#1}}$}\,}

\newcommand{\card}[1]{\lvert #1 \rvert}

\newcommand{\forces}[2]{\Vdash_{#1} \mbox{``} #2 \mbox{''}}

\newcommand{\Poset}{{\mathbb P}}

\newcommand{\Naturals}{{\mathbb N}}

\newcommand{\betan}{\beta{\mathbb N} \setminus {\mathbb N}}
\newcommand{\Nstar}{\Naturals^*}

\newcommand{\pomegaf}{{\mathcal P}(\Naturals)/[\Naturals]^{<\aleph_0}}
\title{More on Tie-points and homeomorphism in $\Naturals^*$}

\author[A. Dow]{Alan Dow}
\address{}
\author[S. Shelah]{Saharon Shelah}
\address{Department of Mathematics, Rutgers University, Hill Center,
 Piscataway, 
 New Jersey, U.S.A. 08854-8019}
\curraddr{Institute of Mathematics\\Hebrew University\\
Givat Ram, Jerusalem 91904, Israel}
\email{shelah@math.rutgers.edu}
\date{\today}
\thanks{
Research of the first author was supported by NSF grant No. NSF-.
The research of the second  author was supported by The Israel Science
Foundation founded by the Israel Academy of Sciences and Humanities, and
by NSF grant No. NSF- . This is paper number 917 
 in the second  author's personal listing}
\begin{document}
\begin{abstract}
A point $x$ is a (bow) tie-point of a space $X$ if $X\setminus \{x\}$
can be partitioned into (relatively) clopen sets each with $x$ in its
closure. We picture (and denote) this as  $X = A\tie{x} B$ where $A,
B$ are the closed sets which have a unique common accumulation point
$x$. Tie-points have appeared in the construction of non-trivial
autohomeomorphisms of $\betan=\Nstar$ (e.g. \cite{veli.oca,ShSt735})
and in the 
recent study \cite{RLevy,DTech1}
 of (precisely) 2-to-1 maps on $\Nstar$. In these cases
the tie-points have been the unique fixed point of an involution on
$\Nstar $. One application of the results in this paper is
 the consistency of there being a 
2-to-1 continuous image of $\Nstar$ which is not a homeomorph of
$\Nstar$. 
\end{abstract}
\bibliographystyle{plain}

\maketitle

\section{Introduction}

A point $x$ is a tie-point of a space $X$ if there are closed sets
$A,B$ of $X$ such that $\{x\}=A\cap B$ and $x$ is an adherent point of
both $A$ and $B$. We let $X = A \tie{x} B$ denote this relation and
say that $x$ is a tie-point as witnessed by $A,B$.  Let $A \equiv_x B$
mean that there is a homeomorphism from 
$A$ to $B$ with $x$ as a fixed point. If $X = A\tie{x} B$ and
$A\equiv_x B$, then there is an involution $F$ of $X$ (i.e. $F^2 = F$)
such that $\{x\} = \fix(F)$. In this case we will say that $x$ is a
symmetric tie-point of $X$. 

An autohomeomorphism $F$ of $\Nstar$ is said to be {\em trivial\/} if
there is a bijection $f$ between cofinite subsets of $\Naturals$ such
that $F = \beta f \restriction \Nstar$.  Since the fixed point 
set of a trivial autohomeomorphism is clopen, a symmetric tie-point
gives rise  to a non-trivial 
autohomeomorphism.

If $A$ and $B$ are arbitrary compact spaces, and if $x\in A$ and $y\in
B$ are accumulation points, then let $A\tie{x{=}y} B$ denote the
quotient space of $A\oplus B$ obtained by identifying $x$ and
$y$ and let $xy$ denote the collapsed point.
 Clearly the point $xy$ is a tie-point of this space. 

In this paper we establish the following theorem.

\begin{theorem} It is consistent\label{main} that $\Nstar$ has 
symmetric tie-points $x,y$ as
  witnessed by $A,B$ and $A',B'$ respectively such that $\Nstar$ is 
not homeomorphic to  the space $ A\tie{x=y}A'$
\end{theorem}

\begin{corollary} It is consistent that there is a 2-to-1 image of
  $\Nstar$ which  is not a homeomorph of $\Nstar$.
\end{corollary}

One can generalize the notion of tie-point and, for a point $x\in
\Nstar$, consider how many disjoint clopen subsets of $\Nstar \setminus
\{x\}$ (each accumulating to $x$) can be found. 
 Let us say
that a tie-point $x$ of $\Nstar$ satisfies  $\tau(x)\geq n$ if
$\Nstar\setminus \{x\}$ can be partitioned into $n$ many disjoint
clopen subsets each accumulating to $x$. Naturally, we will let
$\tau(x)=n$ denote that $\tau(x)\geq n$ and $\tau(x)\not \geq n{+}1$. 
Each point $x$ of character $\omega_1$ in $\Nstar$
is a symmetric tie-point and
satisfies that $\tau(x)\geq n$ for all $n$.
We list several open questions in the
final section.

More generally one could  study the symmetry group of a point $x\in
\Nstar$: e.g. set 
$G_x$ to be  the set of autohomeomorphisms $F$ 
of $\Nstar$ that satisfy $\fix(F)=\{x\}$ and two are identified if
they are the same on some clopen neighborhood of $x$.

\begin{theorem} It is consistent\label{two} that $\Nstar$ has a
  tie-point $x$ such that $\tau(x)=2$  and such that
 with $\Nstar = A\tie{x}B$, 
 neither $A$ nor $B$ is a homeomorph of
 $\Nstar$. In addition, there are no symmetric
 tie-points.  
\end{theorem}

 The following partial order $\Poset_2$,
was introduced by Velickovic in \cite{veli.oca}
to add a non-trivial automorphism of $\pomegaf$ while doing as 
little else as possible --- at least assuming PFA.
\begin{definition}
    The\label{poset} 
partial order $\Poset_2$ is defined to consist of all 1-to-1
    functions 
$f:A\rightarrow B$ where 
\begin{itemize}
\item $A\subseteq\omega$ and $B\subseteq\omega$ 
\item for all $i\in\omega$ and $n\in\omega$, $f(i)\in
  (2^{n+1}\setminus 2^n)$ if and only if 
 $i\in (2^{n+1}\setminus 2^n)$
  \item $\limsup_{n\rightarrow\omega}\card{(2^{n+1}\setminus 2^n)\setminus
A} = \omega$
and hence, by the previous condition,
$\limsup_{n\rightarrow\omega}\card{(2^{n+1}\setminus 2^n)\setminus
B} = \omega$
\end{itemize} 
The ordering on $\Poset_2$ is $\subseteq^{\ast}$.
\end{definition}

We define some trivial generalizations of $\Poset_2$. We use the
notation $\Poset_2$ to signify that this poset introduces an
involution of 
$\Nstar$ because the conditions $g=f\cup f^{-1}$ satisfy that $g^2=g$.  
In the definition of $\Poset_2$ it is possible to suppress mention of 
$A, B$ (which we do)
 and to have the poset $\Poset_2$ consist simply of the
functions $g$ (and to treat $A = \min(g) = \{ i\in \dom(g) : i<g(i)\}$
and to treat $B$ as $\max(g) = \{ i \in \dom(g) : g(i)<i\}$). 

Let $\Poset_1$ denote the poset we get if we omit mention of $f$ but
consisting only of disjoint pairs $(A,B)$, satisying the growth 
condition in Definition \ref{poset},
and extension is coordinatewise
mod finite containment.
 To be consistent with the other two posets, we may
instead represent  the elements of $\Poset_1$ as partial
functions into 2. 

More generally,
 let $\Poset_\ell$ be similar to $\Poset_2$ except that we assume
that conditions consist of functions $g$ satisfying that 
 $\{ i, g(i),g^2(i), \ldots, g^\ell(i)\}$ has precisely $\ell$ elements 
for all $i\in\dom(g)$ (and replace the intervals $2^{n+1}\setminus
2^n$ by $\ell^{n+1}\setminus \ell^n$ in the definition).

The basic properties of $\Poset_2$ as defined by Velickovic and
treated by Shelah and Steprans are also true of $\Poset_\ell$ 
for all $\ell\in\Naturals$.

In particular, for example, it is easily seen that

\begin{proposition} If\label{observe} $L\subset\Naturals$
and  $\Poset^* = \Pi_{\ell\in L} \Poset_\ell$  (with full supports) 
and $G$ is a $\Poset^*$-generic filter,
then in $V[G]$, for each $\ell\in L$, there is a tie-point $x_\ell\in
\Nstar$ 
with $\tau(x_\ell)\geq \ell$. 
\end{proposition}

For the proof of Theorem \ref{main} we use $\Poset_2\times \Poset_2$
and for the proof of Theorem \ref{two} we use $\Poset_1$.

An ideal $\mathcal I$ on $\Naturals$ is said to be ccc over fin
\cite{Farah00}, if 
for each uncountable almost disjoint family, all but countably
many of them are in $\mathcal I$.  An ideal is a $P$-ideal if 
it is countably directed closed mod finite.

The following main result is extracted from 
\cite{step.28} and \cite{step.29} which we record without proof.

\begin{lemma}[PFA] If\label{mainlemma} $\Poset^*$ 
is a finite or countable product (repetitions allowed) 
of  posets from the set $\{ \Poset_\ell : \ell \in \Naturals\}$
and if
 $G$ is a $\Poset^*$-generic filter,
then in $V[G]$ every autohomeomorphism 
of $\Nstar$ has the property that the ideal of sets on which it is
trivial is a $P$-ideal which is ccc over fin.
\end{lemma}

\begin{corollary}[PFA] If\label{maincorollary}
$\Poset^*$ is a 
finite or countable
product of posets from the set $\{\Poset_\ell : \ell\in \Naturals\}$,
and if 
 $G$ is a $\Poset^*$-generic filter,
then in $V[G]$ if $F$ is an 
autohomeomorphism 
of $\Nstar$ and $\{Z_\alpha : \alpha\in \omega_2\}$ is an increasing
mod finite chain of infinite subsets of $\Naturals$, there is an
$\alpha_0\in \omega_2$ and
a collection $\{ h_\alpha : \alpha\in \omega_2\}$ of 1-to-1 
functions such that $\dom(h_\alpha)=Z_\alpha$  and
for all $\beta\in \omega_2$ and $a\subset Z_\beta\setminus
Z_{\alpha_0}$, $F[a]=^* h_\beta[a]$. 
\end{corollary}

Each poset $\Poset^*$ as above 
is $\aleph_1$-closed and $\aleph_2$-distributive
 (see \cite[p.4226]{step.29}). 
In this paper we will restrict out study to finite products.
 The following partial order can be used to show that these
products are $\aleph_2$-distributive.

\begin{definition}
Let $\Poset^*$ be a finite product of posets from
$\{\Poset_\ell : \ell\in \Naturals\}$.
   Given $\{\vec f_{\xi} : \xi\in\mu\}=\mathfrak{F}\subset \Poset^*$
(decreasing in the ordering on $\Poset^*$), define
   $\Poset(\mathfrak{F})$  
to be the partial order consisting of all $g\in
\Poset^*$ such that there is some $\xi\in\mu$ such that
$\vec g\equiv^{\ast} \vec f_{\xi}$. The ordering on  $\Poset(\mathfrak{F})$ 
is coordinatewise
$\supseteq$ as opposed to ${}^*{\supseteq}$ in
$\Poset^*$. 
\end{definition}

\begin{corollary}[PFA] If\label{otherautos}
$\Poset^*$ is a 
finite or countable
product of posets from the set $\{\Poset_\ell : \ell\in \Naturals\}$,
and if 
 $G$ is a $\Poset^*$-generic filter,
then in $V[G]$ if $F$ is an
involution
of $\Nstar$ 
with a unique fixed point $x$, then
$x$ is a $P_{\omega_2}$-point and
$\Nstar = A\tie{x}B$ for some $A,B$ such
that $F[A]=B$.
\end{corollary}

\begin{proof} 
We may assume that $F$ also denotes an arbitrary lifting of $F$ 
to $[\Naturals]^\omega$ in the sense that for each 
$Y\subset\Naturals$, $(F[Y])^* = F[Y^*]$. 
Let $\mathcal Z_x = [\Naturals]^\omega \setminus x$ (the dual ideal to
$x$). For each $Z\in \mathcal Z_x$, $F[Z]$ is also in $\mathcal Z$ and
$F[Z\cup F[Z]] =^* Z\cup F[Z]$. So let us now assume that 
$\mathcal Z$ denotes those $Z\in \mathcal Z_x$ such that $Z=^*F[Z]$.
Given $Z\in \mathcal Z$, since $\fix(F)\cap Z^*=\emptyset$,
there is a collection $\mathcal Y\subset [Z]^\omega$
such that $F[Y]\cap Y=^*\emptyset$ for each $Y\in \mathcal Y$,
and such that $Z^*$ is covered by $\{ Y^* : Y\in \mathcal Y\}$. 
By compactness, we may assume that $\mathcal Y = \{ Y_0, \ldots,
Y_n\}$ is finite. 
 Set $Z_0 = Y_0\cup F[Y_0]$.
By induction, replace $Y_k$ by $Y_k\setminus \bigcup_{j<k}Z_j
$ and define $Z_k =Y_k\cup F[Y_k]$.  Therefore $Y_Z = \bigcup_k Y_k$
satisfies that $Y_Z\cap F[Y_Z]=^* \emptyset$ and $Z=Y_Z\cup
F[Y_Z]$. This shows that for each $Z\in \mathcal Z$ 
 there is a partition of $Z=Z^0\cup Z^1$ such that
$F[Z^0]=^* Z^1$.  It now follows that $x$ is a $P$-point, for if $\{
Z_n = Z_n^0\cup Z_n^1 : n\in \Naturals\}\subset \mathcal Z
$ are pairwise disjoint, then $x\notin \overline{ \bigcup_n Z_n^*}$
since $F[\overline{\bigcup_n (Z_n^0)^*}]=
\overline{\bigcup_n (Z_n^1)^*}$ and $\overline{\bigcup_n (Z_n^0)^*}$
is
disjoint from $\overline{\bigcup_n(Z_n^1)^*}$. 

Now we prove that it is a $P_{\omega_2}$-point. Assume that $\{
Z_\alpha : \alpha\in \omega_1\}\subset \mathcal Z$ is a mod finite
increasing sequence. By Lemma
\ref{mainlemma}  (similar to Corollary \ref{maincorollary}) we 
may assume, by possibly  removing some $Z_{\alpha_0}$ from each 
$Z_\alpha$, that
there is a
sequence $\{ h_\alpha : \alpha \in \omega_1\}$ of involutions 
such that $h_\alpha$ induces $F\restriction Z_\alpha^*$. 
For each $\alpha\in \omega_1$, let $a_\alpha = \min(h_\alpha) 
=\{ i\in \dom(h_\alpha) : i<h_\alpha(i)\}$ and $b_\alpha =
Z_\alpha\setminus a_\alpha$. It follows that $F[a_\alpha]=^*
b_\alpha$. 
Since $\Poset^*$ is $\aleph_2$-distributive, all of these
$\aleph_1$-sized sets are in $V$ which is a model of PFA. 
If $x$ is in the closure of $\bigcup_{\alpha\in \omega_1} Z_\alpha^*$, 
then $x$ is in the closure of each of $\bigcup_\alpha a_\alpha^*$ 
and $\bigcup_\alpha b_\alpha^*$. Therefore, it suffices to 
show that $\mathcal A =
\{ (a_\alpha, b_\alpha) : \alpha\in \omega_1\}$ can not
form a gap in $V$. As is well-known, if $\mathcal A$ does
form a gap, there is a ccc poset $Q_{\mathcal A}$ which adds an
uncountable $I$ such that $\{ (a_\alpha, b_\alpha) : \alpha \in I\}$ 
forms a Hausdorf-gap (i.e. {\em
  freezes\/} the gap). It is easy to prove that if $\mathbb C$ is the
poset for adding $\omega_1$-many almost disjoint Cohen reals,
$\{ \dot C_\xi : \xi\in \omega_1\}$, 
then a similar ccc poset $\mathbb C* \dot
Q$ will introduce, for each $\xi\in \omega_1$,
 an uncountable $I_\xi\subset \omega_1$, such
that $\{ (\dot C_\xi\cap a_\alpha, \dot C_\xi\cap b_\alpha) :
\alpha\in I_\xi\}$ 
is a Hausdorff-gap. But now by Lemma \ref{mainlemma},
 it follows 
 that there is some $\xi\in \omega_1$ such that
$Z=C_\xi\in \mathcal Z$, $F\restriction Z^*$ is trivial
 and for some uncountable
 $I\subset\omega_1$ 
$\{ (Z\cap a_\alpha, Z\cap b_\alpha) : \alpha \in I\}$ forms a 
Hausdorff-gap. This however is a contradiction because
 if $h_Z$ induces
$F\restriction Z^*$, then $\min(h_Z) \cap 
\left((Z\cap a_\alpha)\cup (Z\cap b_\alpha)\right)$ is almost equal to
$a_\alpha$ for all $\alpha\in \omega_1$, i.e. $\min(h_Z)$ would have
to split the Hausdorff-gap.
\end{proof}

The forcing
 $\Poset(\mathfrak  F)$ introduces a tuple $\vec{f}$ which
satisfies $\vec{f}\leq \vec{f_\alpha}$ for $\vec{f_\alpha}\in
 \mathfrak 
 F$ but for the fact that $\vec{f}$ may not be a member of
$\Poset^*$ simply because the domains of the component functions
 are too big. There is a
 $\sigma$-centered poset which will choose an appropriate sequence 
 $\vec{f}^*$ of subfunctions
 of $\vec{f}$ which is a member of $\Poset^*$ and which is
 still below each member of $\mathfrak  F$
  (see \cite[2.1]{step.28}).

A strategic choice of the sequence $\mathfrak F$ will ensure that
$\Poset(\mathfrak F)$ is ccc, but remarkably even more is true.
Again we are lifting results from \cite[2.6]{step.28} and
\cite[proof of Thm. 3.1]{step.29}. This is an innovative factoring
of Velickovic's 
original amoeba forcing poset and seems to
 preserve more properties. Let
$\omega_2^{<\omega_1}$ denote the standard collapse which introduces a
function from $\omega_1$ onto $\omega_2$. 

\begin{lemma} 
Let\label{useful-lemma}
 $\Poset^*$ be a finite product of posets from
$\{\Poset_\ell : \ell\in \Naturals\}$.
In the forcing extension, $V[H]$,
 by $\omega_2^{<\omega_1}$,
  there is a descending sequence $\mathfrak F$ from $\Poset^*$ which
  is $\Poset^*$-generic over $V$ and, for which, 
$\Poset(\mathfrak F)$ is ccc and $\omega^\omega$-bounding.
\end{lemma}

It follows also that $\Poset(\mathfrak F)$ preserves that
$\mathbb R\cap V$ is of 
second category. This was crucial in the proof of 
Lemma \ref{mainlemma}. We can manage with the $\omega^\omega$-bounding
property because we are going to use
Lemma \ref{mainlemma}.  A poset is said to be $\omega^\omega$-bounding
if every new function in $\omega^\omega$ is bounded by some ground
model function.

The following proposition is probably well-known but we do not have a
reference. 

\begin{proposition} Assume\label{bounding}
 that $\mathbb Q$ is a ccc
  $\omega^\omega$-bounding poset and that $x$ is an 
 ultrafilter on
  $\Naturals$. If $G$ is a $ \mathbb Q$-generic filter then there is
  no set $A\subset\Naturals$ such that $A\setminus Y$ is finite for
  all $Y\in x$.
\end{proposition}

\begin{proof} Assume that
 $\{ \dot a_n : n\in \omega\}$ are $\mathbb Q$-names of integers such
that $1\forces{\mathbb Q}{\dot a_n \geq n}$. Let $A$ denote the
$\mathbb Q$-name so that $\forces{\mathbb Q}{A=\{\dot a_n : n\in
  \omega\}}$. 
Since $\mathbb Q$ is
$\omega^\omega$-bounding, there is some $q\in \mathbb Q$ and a
sequence $\{ n_k : k\in \omega\}$
in $V$ such that $q\forces{\mathbb Q}{ n_k \leq 
\dot a_i \leq n_{k+2}~~~\forall i\in [n_k,n_{k+1})}$. 
There is some $\ell\in 3$ such that $Y=\bigcup_{k}
[n_{3k+\ell},n_{3k+\ell+1})$ is a member of $x$. On the other hand,
$q\forces{\mathbb Q}{A\cap [n_{3k+\ell+1},n_{3k+\ell+3})}$ is not empty
  for each $k$. Therefore $q\not\forces{\mathbb Q}
{A\setminus Y\ \mbox{is finite} }$.
\end{proof}

Another interesting and useful general lemma is the following.

\begin{lemma} 
Let\label{also-useful-lemma}
 $\mathcal F\subset \Poset_\ell$  (for any $\ell\in \Naturals$)
  be generic over $V$, then for each $\Poset(\mathfrak
  F)$-name $\dot h \in \Naturals^\Naturals$, either there is an 
$ f\in \mathfrak F$  such that 
$ f\forces{\Poset(\mathfrak F)}{\dot
  h\restriction \dom(f) \notin V}$, or there is an $f\in \mathfrak F$
and an increasing sequence $n_0<n_1<\cdots$ of integers such that
for each $i\in [n_k,n_{k+1})$ and each $g<f$ such that $g$ forces a
value on $\dot h(i)$, $f\cup (g\restriction[n_k,n_{k+1}))$ also forces
 a value on $\dot h(i)$.
\end{lemma}

\begin{proof} 
Given any $f$, perform a standard
 fusion (see \cite[2.4]{step.28} or \cite[3.4]{step.29})
  $f_k,n_k$ by picking $L_k \subset [n_{k+1},n_{k+2})$
(absorbed into $\dom(f_{k+1})$) so that for each partial function $s$
on $n_{k}$ which extends $f_k\restriction n_{k}$, 
if there is some integer $i\geq 
n_{k+1}$ for which no ${<}n_{k}$-preserving
extension of $s\cup f_k$ 
forces a value on $\dot h(i)$, then there is such an integer
in  $\dom(f_{k+1})$. Let $\bar f$ be the fusion and note that 
either $\bar f$ forces that $\dot h\restriction \dom(\bar f)$ is not
in $V$, or it forces that our sequence of $n_k$'s does the job.
Thus, we have proven that for each $f$, there is such a $\bar f$, hence
by genericity, there is such an $\bar f$ in $\mathfrak F$.
\end{proof}

\section{Proof of Theorem \ref{main}}

\begin{theorem}[PFA] If $G$ is a generic filter for
$\Poset^* = \Poset_2\times\Poset_2$, then
there are
symmetric tie-points $x,y$ as
  witnessed by $A,B$ and $C,D$ respectively such that $\Nstar$ is 
not homeomorphic to  the space $ A\tie{x=y}C$
\end{theorem}

Assume that $\Nstar$ is homeomorphic to $A\tie{x=y}C$ and
that $z$ is the $\Poset^*$-name of the ultrafilter that is sent (by
the assumed homeomorphism)
to the point $(x,y)$ in the quotient space $A\tie{x=y}C$.

Further notation: let $\{ a_\alpha : \alpha\in\omega_2\}$ be the
$\Poset_2$-names of the infinite subsets of $\Naturals$ which form
the mod-finite increasing chain whose remainders in $\Nstar$
cover  $A\setminus\{x\}$ and, similarly
let $\{ c_\alpha : \alpha\in \omega_2\}$ be the $\Poset_2$-names
(second coordinates though) which form the chain in $C\setminus
\{y\}$.

If we represent $A\tie{x=y}C$ as a quotient of
$(\Naturals\times 2)^*$, we may assume 
that $F$ is a $\Poset^*$-name of a function from
$[\Naturals]^\omega$ into $[\Naturals\times 2]^\omega$ such that
letting 
$Z_\alpha = F^{-1}(a_\alpha\times\{0\} \cup c_\alpha\times\{1\})$
for each $\alpha\in \omega_2$, then $\{Z_\alpha : \alpha\in
\omega_2\}$ forms the dual ideal to $z$, and $F:[Z_\alpha]^\omega
\rightarrow (a_\alpha\times\{0\}\cup c_\alpha\times\{1\})^\omega$
induces the above
 homeomorphism from $Z_\alpha^*$ onto $(a_\alpha^*\times
\{0\}) \cup (c_\alpha^*\times\{1\})$. 

By Corollary \ref{maincorollary}, we may assume that for each
$\beta\in \omega_2$, there is a bijection 
 $h_\beta$ between some cofinite subset of
$Z_\beta$ and some cofinite subset of
$(a_\beta\times\{0\})\cup
(c_\beta\times\{1\})$ which induces $F\restriction [Z_\beta]^\omega$
(since we can just ignore $Z_{\alpha_0}$ for some fixed $\alpha_0$). 
We will use $F\restriction [Z_\beta]^\omega  = h_\beta$ to mean
that $h_\beta$ induces $F\restriction [Z_\beta]^\omega$.
 Note that by the assumptions,
 for each $\beta\in \omega_2$, there is a $\gamma\in \omega_2$ such
that each of $h_\gamma^{-1}(a_\gamma)\setminus Z_\beta$ and
$h_\gamma^{-1}(c_\gamma)\setminus Z_\beta$ are infinite.

Let $H$ be a generic filter for $\omega_2^{<\omega_1}$,
 and assume that
$\mathfrak F\subset \Poset^*$ is chosen as in Lemma
\ref{useful-lemma}.  In this model, let us use $\lambda$ to denote the
$\omega_2$ from $V$. 
Using the fact that $\mathfrak F$ is
 $\Poset^*$-generic over $V$, we may treat all the functions
$h_\alpha $ $(\alpha\in \lambda)$ as members of $V$ since we can take
 the valuation of all the $\Poset^*$-names using $\mathfrak F$. 
Assume that $\dot h$ is a $\Poset(\mathfrak F)$-name of a finite-to-1
function
 from $\Naturals$ into $\Naturals \times 2$ satisfying that
$h_\alpha\subset^* h$ for all $\alpha\in \lambda$.  We show there is
 no such $\dot h$.  

Since $\Poset(\mathfrak F)$ is $\omega^\omega$-bounding, there is a
increasing sequence of integers $\{ n_k : k\in \omega\}$ and an
$\vec{f_0}=(g_0,g_1)\in \mathfrak F$ such that 
\begin{enumerate}
\item for each $i\in [n_k,n_{k+1})$,
  $\vec{f_0}\forces{\Poset(\mathfrak F)}{
\dot h(i)\in ([0,n_{k+2})\times 2)}$
\item for each $i\in [n_k,n_{k+1})$, 
  $\vec{f_0} \forces{\Poset(\mathfrak F)}{
\dot h^{-1}(\{i\}\times 2)\subset [0,n_{k+2})}$
\item for each $k$ and each $j\in \{0,1\}$
 there is an  $m$ such that $n_k<
2^m<2^{m+1}<n_{k+1}$, and
$[2^m,2^{m+1}) \setminus \dom g_j$ has at least $k$ elements.
\end{enumerate}

Choose any $(g_0',g_1')=\vec{f_1}<\vec{f_0}$ such that 
$\Naturals\setminus \dom(g_0') \subset \bigcup_{k}
[n_{6k+1},n_{6k+2})$ and $\Naturals \setminus \dom(g_1')\subset
\bigcup_{k} [n_{6k+4},n_{6k+5})$. 
Next, choose any $\vec{f_2}<\vec{f_1}$ and some $\alpha\in \lambda$
such that $\vec{f_2}\forces{\Poset(\mathfrak F)}{
\dom(g_0')\subset^* a_\alpha\cup g_0'[a_\alpha]\ \mbox{and}\
\dom(g_1')\subset^* c_\alpha\cup g_1'[c_\alpha]}$. 
For each $\gamma\in \lambda$, 
note that $\vec{f_2}\forces{\Poset(\mathfrak F)}{
a_\gamma \setminus a_\alpha \subset^* \Naturals\setminus \dom(g_0')}$
and similarly
$\vec{f_2}\forces{\Poset(\mathfrak F)}{
c_\gamma\setminus c_\alpha\subset^* \Naturals \setminus \dom(g_1')}$.

Now consider the two disjoint sets:
$Y_0 = \bigcup_{k} [n_{6k},n_{6k+3}) $
and $Y_1=\bigcup_{k} [n_{6k+3},n_{6k+6})$. 
Since $z$ is an ultrafilter in this extension,
by possibly extending $\vec{f_2}$ even more, we may assume that there
is some $j\in \{0,1\}$ and
some $\beta>\alpha$ such that $\vec{f_2}\forces{\Poset(\mathfrak
  F)}{ Y_j\subset^* Z_\beta}$. Without loss of generality (by
symmetry) we may assume that $j=0$. Consider any $\gamma\in \lambda$.
Since we are assuming that $h_\gamma\subset^* \dot h$, we have that
 $\vec{f_2}$ forces that $h_\gamma[Z_\gamma\setminus
 Z_\alpha] =^* \dot h[Z_\gamma\setminus Z_\alpha]$. We also have
that $\vec{f_2}\forces{\Poset(\mathcal F)}
{ \dot h[Y_0]\ {{}^*{\supset}}\ 
(a_\gamma\setminus a_\alpha)\times \{0\} =^*
h_\gamma[Z_\gamma\setminus Z_\alpha]\cap \Naturals\times\{0\}
}
$. 
Putting this all together, 
we now have that
 $\vec{f_2}$ forces that $\dot h[Z_\beta]$ almost contains
 $(a_\gamma\setminus a_\alpha)\times \{0\}$ for all $\gamma\in
 \lambda$; which clearly contradicts that $\dot h[Z_\beta]$ is
 supposed to be almost equal to $h_\beta[Z_\beta]$.

So now what?  Well, let $H_2$ be a generic filter for
$\Poset(\mathfrak F)$ and consider the family of functions
$\mathcal H_\lambda =
\{ h_\alpha  : \alpha\in \lambda\}$ which we know does not have a
common finite-to-1 extension. 

 Before proceeding, we need to show that $\mathcal H_\lambda$ 
 does not have any extension $h$.  
If $\dot h$ is any $\Poset(\mathfrak F)$-name of a
 function for which it is forced that
$h_\alpha\subset^* \dot h$ for all $\alpha\in \lambda$, then 
there is some $\ell\in \Naturals$ such that 
$\dot Y =  h^{-1}(h(\ell))$ is (forced to be) 
infinite. It follows easily that $\dot Y$ is forced to be almost 
contained in every member of $z$. By Lemma \ref{bounding} this cannot
happen. Therefore the family $\mathcal H_\lambda$ does not have any
common extension.

Given such a family as $\mathcal H_\lambda$, 
there is a well-known 
proper poset $Q_1$ (see \cite[3.1]{DoSiVa},
\cite[2.2.1]{Farah00}, and \cite[p9]{veli.oca})
which will force an uncountable cofinal $I\subset \lambda$ and a
collection of integers $\{ k_{\alpha,\beta} : \alpha<\beta\in I\}$
satisfying that $h_\alpha(k_{\alpha,\beta}) \neq
h_\beta(k_{\alpha,\beta})$ (and both are defined) for $\alpha<\beta\in
I$. So, let $\dot Q_1$ be the $\omega_2^{<\omega}* \Poset(\mathfrak
F)$-name of the above  mentioned poset. In addition, let $\dot \varphi
$ be the 
 $\omega_2^{<\omega}* \Poset(\mathfrak F) * \dot Q_1$-name of the
enumerating function from $\omega_1$ onto $I$, and let $\dot
k_{\alpha,\beta}$ (for $\alpha<\beta\in \omega_2$) be the name
of the integer $k_{\dot \varphi(\alpha),\dot \varphi(\beta)}$. Thus
for each $\alpha<\beta\in \omega_1$, there is a dense set
$D(\alpha,\beta)\subset
\omega_2^{<\omega}* \Poset(\mathfrak F) * \dot Q_1$
such that for each member $p$ of $D(\alpha,\beta)$,
there are functions $h_\alpha, h_\beta$ in $V$ and sets 
$Z_\alpha=\dom(h_\alpha), Z_\beta=\dom(h_\beta)$
and integers $k=k(\alpha,\beta)\in Z_\alpha\cap Z_\beta$
 such that
$$p\forces{\omega_2^{<\omega}*\Poset(\mathfrak F)*\dot Q_1}
{F\restriction [Z_\alpha]^\omega = h_\alpha,
F\restriction [Z_\beta]^\omega = h_\beta,
h_\alpha(k)\neq h_\beta(k)
}.
$$

Finally, let $\dot Q_2$ be the
 $\omega_2^{<\omega}* \Poset(\mathfrak F) * \dot Q_1$-name of the
$\sigma$-centered poset which forces an element $\vec{f}\in \Poset^*$
which is below every member of $\mathfrak F$. 
Again, there is a countable collection of dense subsets of the proper
poset 
 $\omega_2^{<\omega}* \Poset(\mathfrak F) * \dot Q_1 *\dot Q_2$ which
 determine the values of $\vec{f}$. 

Applying PFA to the above proper poset and the family of $\omega_1$
mentioned dense sets, we find there is a sequence 
$\{ h'_\alpha, Z'_\alpha : \alpha\in \omega_1\}$,  integers
$\{ k_{\alpha,\beta} : \alpha<\beta\in \omega_1\}$, and a condition
$\vec{f}\in \Poset^*$ such that, for all $\alpha<\beta$ and
$k=k(\alpha,\beta)$, 
$$\vec{f}\forces{\Poset^*}
{F\restriction [Z'_\alpha]^\omega = h'_\alpha,
F\restriction [Z'_\beta]^\omega = h'_\beta,
h'_\alpha(k)\neq h'_\beta(k)
}.
$$

But, we also know that we can choose $\vec{f}$ so that 
there is some $\lambda\in \omega_2$, and some $h_\lambda, Z_\lambda$ 
such that, for all $\alpha\in \omega_1$,
 $Z_\alpha'\subset^*
Z_\lambda$ and $F\restriction [Z_\lambda]^\omega = h_\lambda$. 

It follows of course that for all $\alpha\in \omega_1$, there is some
$n_\alpha$ such that $h_\alpha'\restriction [n_\alpha,\omega)\subset
h_\lambda$. Let $J\in [\omega_1]^{\omega_1}$, $n\in \omega$, and
$h'$ a function with $\dom(h')\subset n$
 such
that
$n_\alpha = n$ and $h_\alpha'\restriction n = h'$ for all $\alpha\in
J$. We now have a contradiction since if $\alpha<\beta\in J$ then
clearly $k=k(\alpha,\beta)\geq n$ and this contradicts that 
$h'_\alpha(k)$ and $h'_\beta(k)$ are both supposed to equal
$h_\lambda(k)$.

\section{proof of Theorem \ref{two}}

\begin{theorem}[PFA]
If\label{pfatwo}
 $G$ is a generic filter for $\Poset_1$, then
a   tie-point $x$ is introduced such that $\tau(x)=2$  and 
 with $\Nstar = A\tie{x}B$, 
 neither $A$ nor $B$ is a homeomorph of
 $\Nstar$. In addition, there is no involution $F$ on
$\Nstar$ which has a unique fixed point, and so, no tie-point
is symmetric.
\end{theorem}

Assume that $V$ is a model of PFA and that $\Poset=\Poset_1$. The
elements of $\Poset$ are partial functions $f$ from $\Naturals$
 into 2 which also
satisfy that $\limsup_{n\in \Naturals} |2^{n{+}1}\setminus (2^n\cup
\dom(f))| =\infty$. The ordering on $\Poset$ is
that
 $f<g$ ($f,g\in \Poset$) if $g\subset^* f$. For each $f\in
\Poset$, let $a_f=f^{-1}(0)$ and $b_f=f^{-1}(1)$. 

Again we assume that $\{ a_\alpha : \alpha \in \omega_2\}$ is the
sequence of $\Poset$-names satisfying that $\Nstar = A\tie{x}B$ and
 $A\setminus \{x\} =\bigcup \{ a_\alpha^* : \alpha\in \omega_2\}$. 
Of course by this we mean that for each $f\in G$, there are $\alpha\in
\omega_2$, $a\in [\Naturals]^\omega$,
 and $f_1\in G$ such that $a_{f} \subset^* a \subset a_{f_1}$
and $f_1\forces{\Poset} { \check a = a_\alpha}$. 

Next we assume that, if $A$ is homeomorphic to 
$\Nstar$, then $F$ is a $\Poset$-name of a homeomorphism from $\Nstar$
to $A$ and let $z$ denote the point in $\Nstar$ which $F$ sends to
$x$. Also, let $Z_\alpha$ be the $\Poset$-name of $F^{-1}[a_\alpha]$
and recall that $\Nstar\setminus \{z\} = \bigcup\{Z_\alpha^* :
\alpha\in \omega_2\}$.
As above, we may also
assume that for each $\alpha\in \omega_2$, there is a $\Poset$-name of
a function $h_\alpha$ with $\dom(h_\alpha)=Z_\alpha$ such that 
$F\restriction [Z_\alpha]^\omega$ is induced by $h_\alpha$.

Furthermore if $\tau(x)>2$, then one of  $A\setminus \{x\}$ or
$B\setminus\{x\}$  can be
partitioned into disjoint clopen non-compact sets. We may assume that
it is $A\setminus \{x\}$ which can be so partitioned. 
Therefore 
there is some sequence $\{ c_\alpha :
\alpha\in \omega_2\}$ of $\Poset$-names such that for each
$\alpha<\beta\in \omega_2$, 
$c_\beta\subset a_\beta$ and $c_\beta\cap a_\alpha =^* c_\alpha$. In
addition, for each $\alpha<\omega_2$ there must be a
$\beta\in\omega_2$ such that $c_\beta\setminus a_\alpha$ and
$a_\beta\setminus (c_\beta\cup a_\alpha)$ are both infinite.

Now assume that $H$ is $\omega_2^{<\omega_1}$-generic and again 
choose a sequence $\mathfrak F\subset \Poset$ which is $V$-generic for
$\Poset$ and which forces that $\Poset(\mathfrak F)$ is ccc and
$\omega^\omega$-bounding.  For the rest of the proof we work
in the model $V[H]$ and we again let $\lambda$ denote the ordinal
$\omega_2^V$.

In the case of $\Poset_1$ we are able to prove a significant
strengthening of Lemma \ref{also-useful-lemma}.

\begin{lemma} Assume\label{homog}
 that $\dot h$ is a $\Poset(\mathfrak F)$-name of
  a function from $\Naturals$ to $\Naturals$. 
Either there is an 
$ f\in \mathfrak F$ and such that 
$ f\forces{\Poset(\mathfrak F)}{\dot
  h\restriction \dom(f) \notin V}$, or there is an $f\in \mathfrak F$
and an increasing sequence $ m_1 < m_2 < \cdots $
 of integers such that $\Naturals \setminus \dom(f) = \bigcup_k S_k$
 where
$S_k\subset 2^{m_{k+1}}\setminus  2^{m_k}$ and for each $i\in S_k$ the
condition $f\cup \{(i,0)\}$ forces a value on $\dot h(i)$. 
\end{lemma}

\begin{proof}
First we  choose $f_0\in \mathfrak F$ 
and some increasing sequence $n_0<n_1<\cdots n_k<\cdots $
 as in  
 Lemma \ref{also-useful-lemma}.  We may choose, for each $k$,
an $m_k$ such that $n_k \leq 2^{m_k} < 2^{m_k+1} \leq n_{k+1}$ 
such that $\limsup_{k} | 2^{m_k+1}\setminus (2^{m_k}\cup \dom(f_0))| =
\infty$. For each $k$, let $S^0_k = 
2^{m_k+1}\setminus (2^{m_k}\cup \dom(f_0))$.
By re-indexing we may assume that $|S^0_k|
\geq k$, and we may arrange that $\Naturals \setminus \dom(f_0)
$ is equal to $\bigcup_k S^0_k$ and set $L_0=\Naturals$. 
For each $k\in L_0$,  let $i^0_k = \min S^0_k$ and choose any
$f'_1<f_0$ such that (by definition of $\Poset$)  
 $I_0= \{ i^0_k : k\in L_0\} \subset
(f_1')^{-1}(0)$ and (by assumption on $\dot h$)
 $f_1'$ forces a value on $\dot h(i^0_k)$ for each $k\in L_0$. 
Set $f_1 = f_1'\restriction (\Naturals \setminus I_0)$ and
for each $k\in L_0$, let $S^1_k = S^0_k \setminus
(\{i^0_k\}\cup\dom(f_1))$. By further extending $f_1$ we may also
assume that  $f_1\cup \{(i^0_k,1)\}$ also forces a value on $\dot
h(i^0_k)$.
Choose $L_1\subset L_0$ such that $\lim_{k\in L_1} |S^1_k| =
\infty$. Notice that each member of $i^0_k $ is the minimum element of 
$S^1_k$.
Again, we may extend $f_1$ and assume that 
$\Naturals \setminus \dom(f_1)$ is equal to $\bigcup_{k\in L_1} S^1_k$.
Suppose now we have some
infinite $L_j$, some $f_j$, and for $k\in L_j$, an increasing
sequence $\{ i^0_k, i^1_k, \ldots, i^{j-1}_k\}\subset S^0_k$. Assume
further that 
$$S^j_k\cup \{ i^\ell_k : \ell<j\} = S^0_k \setminus \dom(f_j) $$
and that $\lim_{k\in L_j} |S^j_k| = \infty$. For each $k\in L_j$, 
let $i^j_k = \min(S^j_k\setminus \{i^\ell_k : \ell<j\})$. By a simple
recursion of length $2^j$, there is an $f_{j+1} < f_j$ such that,
for each $k\in L_j$, 
$\{i^\ell_k : \ell \leq j\}\subset S_k^0\setminus \dom(f_{j+1})$ and
for each function $s$ from $\{i^\ell_k : \ell \leq j\}$ into $2$, 
the condition $f_{j+1}\cup s$ forces a value on $\dot h(i^j_k)$.
Again find $L_{j+1}\subset L_j$ so that $\lim_{k\in L_{j+1}} |
S^{j+1}_k| =\infty$ (where $S^{j+1}_k = S^0_k\setminus \dom(f_{j+1})$)
and extend $f_{j+1}$ so that $\Naturals \setminus \dom(f_{j+1})$ is
equal to $\bigcup_{k\in L_{j+1}} S^{j+1}_k$. 

We are half-way there. At the end of this fusion, the function
 $\bar f = \bigcup_j f_j$ is a member of $\Poset$ because for each $j$
 and 
 $k\in L_{j+1}$, $2^{m_k+1}\setminus (2^{m_k}\cup \dom(\bar f))\supset
\{i^0_k, \ldots, i^j_k\}$.  
 For each $k$, let $\bar S_k = S^0_k
\setminus \dom(\bar f)$ and, by possibly extending $\bar f$, we
may again assume that there is some $L$ such that $\lim_{k\in L} |\bar
S_k| =\infty$ and that, for $k\in L$,
$\bar S_k = \{i^0,i^1_k,\ldots, i^{j_k}_k \}$ for some $j_k$.
What we have proven about $\bar f$ is that it satisfies that for each
$k\in L$ and each $j<j_k$ and each function $s$ from $\{i_k^0,\ldots,
i_k^{j-1}\}$ to 2, $\bar f\cup s \cup (i^j_k,0)$ forces a value on $\dot
h(i^j_k)$. 

To finish, simply repeat the process except this time choose maximal
values and work down the values in $\bar S_k$. Again, by genericity of
 $\mathfrak F$, there must be such a condition as $\bar f$ in
 $\mathfrak F$.
\end{proof}

Returning to the proof of Theorem \ref{pfatwo}, we are ready to use
Lemma \ref{homog} to show that forcing with $\Poset(\mathfrak F)$ will
not introduce undesirable functions $h$ analogous to the argument in
Theorem \ref{main}. 
Indeed, assume that we are in the case that $F$ is
a homeomorphism from $\Nstar$ to $A$ as above, and that $\{ h_\alpha :
\alpha\in \lambda\}$ is the family of functions as above. If we show
that $\dot h$ 
does not satisfy that $h_\alpha\subset^* \dot h$ for each
$\alpha\in \lambda$, then we proceed just as in Theorem \ref{main}. By
Lemma \ref{homog}, we have the condition $f_0\in \mathfrak F$ and the
sequence $S_k$ ($k\in \Naturals$) such that $\Naturals \setminus
\dom(f_0) = \bigcup_k S_k $ and that
 for each $i\in \bigcup_k S_k$, $f_0\cup \{(i,0)\}$ forces a value (call
 it $\bar h(i)$) on $\dot h(i)$. Therefore, $\bar h$ is a function
 with domain $\bigcup_k S_k$ in
 $V$.   It suffices to find a condition in $\Poset$ below $f_0$ 
which forces that there is some $\alpha$ such that $h_\alpha $ is
not extended by $\dot h$.  It is useful to note that if
$Y\subset \bigcup_k S_k$ is such that $\limsup |S_k\setminus Y|$ is
infinite, then for any function $g\in 2^Y$, $f_0\cup g\in \Poset$.

We first check that $\bar h$ is 1-to-1 on a cofinite
 subset. If not, there is an infinite set of pairs $E_j\subset
 \bigcup_k \bar  S_k$,  $\bar h[E_j]$ is a singleton and
such that for each $k$, $\bar S_k\cap
 \bigcup_j E_j$ has at most two
 elements. If $g$ is the function with $\dom(g)=\bigcup_j E_j$ 
which is constantly 0, then $f_0\cup g$ forces that $\dot h$ agrees
with $\bar h$ on $\dom(g)$ and so is not 1-to-1.
On the other hand,  this contradicts that
there is
$f_1 < f_0\cup g$
such that for some $\alpha\in \omega_2$, $a_\alpha$ 
almost contains $(f_0\cup g)^{-1}(0)$
and the
1-to-1 function $h_\alpha$ with domain $a_\alpha$
is supposed to also agree with $\dot h$ on $\dom(g)$. 

But now that we know that $\bar h$ is 1-to-1 
we may choose any $f_1\in
\mathfrak F $ such that $f_1<f_0$ and 
such that there is an $\alpha\in
\omega_2$ with $f_0^{-1}(0)\subset a_\alpha\subset f^{-1}_1(0)$,
and $f_1$ has decided the function $h_\alpha$.
Let $Y$ be any infinite subset of $\Naturals\setminus \dom(f_1)$
which meets each
$\bar S_k$ in at most a single point. If $\bar h[Y]$ meets
 $Z_\alpha$ in an infinite set, then
 choose $f_2<f_1$ so that $f_2[Y]=0$ and there is a $\beta>\alpha$
such that $Y\subset a_\beta$. In this case we will have that
$f_2$ forces that $Y\subset a_\beta\setminus a_\alpha$,
$\dot h\restriction Y \subset^* h_\beta$, and $h_\beta[Y]\cap 
h_\beta[a_\alpha]$ is infinite (contradicting that $h_\beta$ is 1-to-1).
Therefore we must have that $\bar h[Y]$ is almost disjoint from
$Z_\alpha$. Instead consider $f_2<f_1$ so that $f_2[Y]=1$. 
By extending $f_2$ we may assume that there is a $\beta<\omega_2$ such
that $f_2\forces{\Poset(\mathfrak  F)}{Z_\beta \cap \bar h[Y] 
\ \mbox{is infinite}}$. However,
since
$f_2\forces{\Poset(\mathfrak F)}{h_\beta \subset^* \dot h}$, 
 we also have that 
$f_2\forces{\Poset(\mathfrak F)}{h_\beta\restriction (a_\beta\setminus
  a_\alpha) \subset^* \bar h \ \mbox{and}\ 
(a_\beta\setminus a_\alpha)\cap Y=^*\emptyset}$
contradicting that $\bar h$ is 1-to-1 on $\dom(f_2)\setminus
\dom(f_0)$.   
This finishes the proof  that there is no $\Poset(\mathfrak F)$ name
of a function extending all the $h_\alpha$'s ($\alpha\in \lambda$) and
the proof that $F$ can not exist continues as in Theorem \ref{main}. 

Next assume that we have a family $\{c_\alpha : \alpha \in \lambda\}$
as described above and suppose that $C=\dot h^{-1}(0)$ satisfies that
(it is forced)
$C\cap a_\alpha =^* c_\alpha$ for all $\alpha\in \lambda$. If we can
show there is no such $\dot h$, then we will know that in the
extension obtained by forcing with $\Poset(\mathcal F)$, the
collection
 $\{ (c_\alpha, (a_\alpha\setminus c_\alpha)) : \alpha\in \lambda\}$ 
forms an $(\omega_1,\omega_1)$-gap and we can use a proper poset $Q_1$
to 
``freeze'' the gap. Again, meeting $\omega_1$ dense subsets of the
iteration $\omega_2^{<\omega_1}*\Poset(\mathcal F)*Q_1 * Q_2$ (where
$Q_2$ is the $\sigma$-centered poset as in Theorem \ref{main})
introduces a condition $f\in \Poset$ which forces that $c_\lambda$
will not exist.  So, given our name $\dot h$, we repeat the steps
above up to the point where we have $f_0$ and the sequence $\{ S_k :
k\in \Naturals\}$ so that $f_0\cup \{(i,0)\}$ forces a value $\bar
h(i)$ on $\dot h(i)$ for each $i\in \bigcup_k S_k$ and
$\Naturals\setminus \dom(f_0) = \bigcup_k S_k$. Let $Y=\bar h^{-1}(0)$ 
and $Z=\bar h^{-1}(1)$ (of course we may assume that $\bar h(i)\in 2$
for all $i$). Since $x$ is forced to be an ultrafilter, there is an 
$f_1<f_0$ such that 
$\dom(f_1)$ contains one of $Y$ or $Z$. 
If $\dom(f_1)$ contains
$Y$, then $f_1$ forces that $\dot h[a_\beta \setminus \dom(f_1)] = 1$
and so $(a_\beta\setminus \dom(f_1))\subset^* (\Naturals\setminus C)$ 
for all $\beta\in \omega_2$. While if $\dom(f_1)$ contains $Z$, 
then $f_1$ forces that $\dot h[a_\beta \setminus \dom(f_1)]=0$,
 and so $(a_\beta\setminus \dom(f_1))\subset^* C$  for all
$\beta\in \omega_2$. However, taking $\beta$ so large that
each of $c_\beta\setminus \dom(f_1)$ and $(a_\beta\setminus
(c_\beta\cup \dom(f_1))$ are infinite shows that no such 
$\dot h$ exists.

Finally we show that there are no involutions on $\Nstar$ 
which have a unique fixed point. Assume that $F$ is such an involution
and  that $z$ is the unique fixed point of $F$. 
Applying Corollaries \ref{maincorollary} and \ref{otherautos}, we may
assume that $\Nstar\setminus \{z\} = \bigcup_{\alpha\in \omega_2} Z_\alpha^*$
 and that for each $\alpha$, $F\restriction Z_\alpha^*$ is induced by
an involution $h_\alpha$. 

Again let $H$ be $\omega_2^{<\omega_1}$-generic, $\lambda=\omega_2^V$,
 and $\mathfrak
F\subset
\Poset_1$ be $\Poset_1$-generic over $V$. Assume that $\dot h$ is
a $\Poset(\mathcal F)$-name of a function from $\Naturals$ into
$\Naturals$. It suffices to show that no $f\in \mathfrak F$ forces
that $\dot h$ mod finite extends each $h_\alpha$ ($\alpha\in
\lambda$).  

At the risk of being too incomplete, we leave to the reader the fact
that Lemma \ref{also-useful-lemma} can be generalized to show that 
there is an $f\in\mathfrak F$ such that
either $f\forces{\Poset_1}{\dot h\restriction Z_\alpha\notin V}$, or
 there is a sequence $\{ n_k : k\in \Naturals\}$ as before. This is
 simply due to the fact that the $\Poset_1$-name
of the ultrafilter $x_1$ can be replaced by any $\Poset_1$-name of an
ultrafilter on $\Naturals$. Similarly, Lemma \ref{homog} can be
generalized in this setting to establish that there must be an $f\in
\mathfrak F$ and a sequence of sets $\{ m_k, S_k, T_k : k\in K\in 
[\Naturals]^\omega\}$
with bijections $\psi: S_k \rightarrow T_k$ such that
$S_k\subset (2^{m_k+1}\setminus 2^{m_k})\subset[n_k,n_{k+1})$, 
$T_k\subset [n_k,n_{k+1})$, 
$\Naturals\setminus \dom(f)\subset \bigcup_k S_K$,
 and
for each $k$ and $i\in S_k$ and $\bar f<f$
$\bar f$ forces a value on $\dot
h(\psi(i))$ iff $i\in\dom(\bar f)$. 
The difference here is that we may  have
that $f\forces{\Poset_1}{\dom(f)\subset Z_\alpha}$, but there will
be some values of $\dot h$ not yet decided since $V[H]$ does not have
a function extending all the $h_\alpha$'s.   
Set $\Psi = \bigcup \psi$ which is  a 1-to-1 function.

The contradiction now is that there will be some $f'<f$ such that
$f'\forces{\Poset_1}{\Psi^*(x)\neq z}$ (because we know that $x$ is
not a tie-point). Therefore we may assume that $\Psi(\dom(f')\cap 
\dom(\Psi))$ is a member of $z$ and so that $\Psi(\dom(\Psi)\setminus
\dom(f')) $ is not a member of $z$. By assumption, there is some
$\bar f<f'$ and an $\alpha\in \lambda$ such that
$\bar f\forces{\Poset_1}{\Psi(\dom(\Psi)\setminus \dom(f'))\subset
  Z_\alpha}$.  However this implies $\bar f$ forces that
$\dot h(\Psi(i)) = h_\alpha(\Psi(i))$ for almost all $i\in \bigcup_k
S_k\setminus \dom(\bar f)$, contradicting that $\bar f$ does not
force a value on $\bar h(\Psi(i))$ for all $i\notin \bar f$.

\section{questions}

\begin{question} Assume PFA. If $G$ is $\Poset_2$-generic, and $\Nstar
  = A\tie{x}B$ is the generic tie-point introduced by $\Poset_2$, is
  it true that $A$ is not homeomorphic to $\Nstar$? Is 
it true that $\tau(x)=2$? Is it true that each tie-point is a
symmetric tie-point?
\end{question}

\begin{remark} The tie-point $x_3$ 
introduced by $\Poset_3$ does not satisfy
  that $\tau(x_3)=3$.  This can be seen as follows. 
For each $f\in \Poset_3$, we can partition $\min(f)$ into
 $\{ i\in \dom(f) : i< f(i)<f^2(i)\}$ and $\{i \in \dom(f) : 
 i< f^2(i) < f(i)\}$. 
\end{remark}

It seems then that the tie-points $x_\ell$
introduced by $\Poset_\ell$ might be
better characterized by the property that there is an
autohomeomorphism $F_\ell$ of $\Nstar$ satisfying that
$\fix(F_\ell)=\{x_\ell\}$, and each
$y\in \Nstar\setminus \{x\}$ has an orbit of size $\ell$.

\begin{remark} A small modification to the poset $\Poset_2$ will 
result in a tie-point $\Nstar = A\tie{x}B$ such that $A$ (hence the
quotient space by the associated involution) is homeomorphic to
$\Nstar$. The modification is to build into the conditions a map
from the pairs $\{i,f(i)\}$ into $\Naturals$. A natural way to do this
is the poset $f\in \Poset_2^+ $ if $f$ is a 2-to-1 function such that
for each $n$, $f$ maps $\dom(f)\cap (2^{n+1}\setminus 2^n)$ into 
$2^{n}\setminus 2^{n-1}$, and again
$\limsup_{n} |2^{n+1}\setminus (\dom(f)\cup
2^n)|=\infty$. $\Poset_2^+$ is ordered by almost containment.
 The generic
filter introduces an $\omega_2$-sequence $\{ f_\alpha : \alpha \in
\omega_2\}$ and two ultrafilters:
 $x\supset \{\Naturals \setminus \dom(f_\alpha) :
\alpha \in \omega_2\}$ 
and $z\supset \{ \Naturals \setminus \ran(f_\alpha) : 
\alpha\in \omega_2\}$. 
For each $\alpha$ and $a_\alpha = \min(f_\alpha) = \{ i\in
\dom(f_\alpha) : i = \min( f_\alpha^{-1}(f_\alpha(i))\}$, we set
$A=\{x\}\cup \bigcup_{\alpha} a_\alpha^*$ and $B=\{x\}\cup
\bigcup_\alpha (\dom(f_\alpha)\setminus a_\alpha)^*$, and we have
that $\Nstar =A\tie{x}B$ is a symmetric tie-point. Finally, we
have that $F : A \rightarrow \Nstar$ defined by $F(x)=z$
and $F\restriction A\setminus \{x\} = 
\bigcup_\alpha (f_\alpha)^*$
is a homeomorphism.
\end{remark}

\begin{question} Assume PFA.
If $L$ is a finite subset of $\Naturals$ and
 $\Poset_L = \Pi\{ \Poset_\ell : \ell\in L\}$,
is it true that in 
$V[G]$ that if  $x$ is tie-point, then
 $\tau(x)\in L$; and if $1\notin L$, then every tie-point is a
 symmetric tie-point?
\end{question}


\begin{thebibliography}{1}

\bibitem{DoSiVa}
Alan Dow, Petr Simon, and Jerry~E. Vaughan, \emph{Strong homology and the
  proper forcing axiom}, Proc. Amer. Math. Soc. \textbf{106} (1989), no.~3,
  821--828. \MR{MR961403 (90a:55019)}

\bibitem{DTech1}
Alan Dow and Geta Techanie, \emph{Two-to-one continuous images of {$\Bbb N\sp
  *$}}, Fund. Math. \textbf{186} (2005), no.~2, 177--192. \MR{MR2162384
  (2006f:54003)}


\bibitem{Farah00}
Ilijas Farah, \emph{Analytic quotients: theory of liftings for quotients over
  analytic ideals on the integers}, Mem. Amer. Math. Soc. \textbf{148} (2000),
  no.~702, xvi+177. \MR{MR1711328 (2001c:03076)}

\bibitem{RLevy}
Ronnie Levy, \emph{The weight of certain images of {$\omega$}}, Topology Appl.
  \textbf{153} (2006), no.~13, 2272--2277. \MR{MR2238730 (2007e:54034)}


\bibitem{step.15}
S.~Shelah and J.~Stepr\={a}ns.
\newblock Non-trivial homeomorphisms of $\beta {N}\setminus {N}$ without the
  {C}ontinuum {H}ypothesis.
\newblock {\em Fund. Math.}, 132:135--141, 1989.

\bibitem{step.28}
S.~Shelah and J.~Stepr\={a}ns.
\newblock Somewhere trivial autohomeomorphisms.
\newblock {\em J. London Math. Soc. (2)}, 49:569--580, 1994.


\bibitem{ShSt735}
Saharon Shelah and Juris Stepr{\=a}ns, \emph{Martin's axiom is
  consistent with 
  the existence of nowhere trivial automorphisms}, Proc. Amer. Math. Soc.
  \textbf{130} (2002), no.~7, 2097--2106 (electronic). \MR{1896046
  (2003k:03063)}

\bibitem{step.29}
Juris Stepr{\=a}ns, 
\emph{The autohomeomorphism group of the \v {C}ech-{S}tone
  compactification of the integers}, 
Trans. Amer. Math. Soc. \textbf{355}
  (2003), no.~10, 4223--4240 (electronic). \MR{1990584 (2004e:03087)}


\bibitem{veli.def}
B.~Velickovic.
\newblock Definable automorphisms of ${\mathcal P}(\omega)/fin$.
\newblock {\em Proc. Amer. Math. Soc.}, 96:130--135, 1986.

\bibitem{veli.oca}
Boban Veli{\v{c}}kovi{\'c}.
\newblock ${\rm {O}{C}{A}}$ and automorphisms of ${\mathcal P}(\omega)/{\rm
  fin}$.
\newblock {\em Topology Appl.}, 49(1):1--13, 1993.

\end{thebibliography}
\end{document}